\documentclass{siamltex}
\usepackage{ams}
\newcommand{\sgn}{{\rm sgn}\kern 0.12em}
\newcommand{\argmin}{{\rm argmin}\kern 0.12em}

\newcommand{\cL}{{\mathcal L}}

\newcommand{\cK}{\mathcal K}

\newcommand{\bd}{{\rm bd}\kern 0.12em}
\newcommand{\inte}{{\rm int}\kern 0.12em}

 \usepackage{graphicx}

\title{A characterization of the behavior of the Anderson acceleration on linear problems}

\author{Florian A. Potra\thanks{Department of Mathematics and Statistics,
University of Maryland Baltimore County, 1000 Hilltop Circle,
Baltimore, MD 22150, USA ({\tt potra@umbc.edu}). The work of
this author was supported by the National Science Foundation under
Grant No. 0728878.}\and Hans Engler\thanks{Department of Mathematics,
Georgetown University,
Box 571233,
Washington, D.C. 20057,
U.S.A.({\tt engler@georgetown.edu
})}}

\begin{document}

\maketitle

\begin{abstract}
We give a complete characterization of the behavior of the Anderson acceleration (with arbitrary nonzero mixing parameters) on linear problems. Let $\nu$ be the grade of the residual at the starting point with respect to the matrix defining the linear problem.  We show that if Anderson acceleration does not stagnate (that is, produces different iterates) up to $\nu$, then the sequence of its iterates converges to the exact solution of the linear problem. Otherwise, the Anderson acceleration converges to the wrong solution.  Anderson acceleration and of GMRES are essentially equivalent up to the index where the iterates of Anderson acceleration begin to stagnate.  This result holds also for  an optimized version of  Anderson acceleration, where at each step the mixing parameter is chosen so that it minimizes the residual of the current iterate.
\end{abstract}



\section{Introduction}\label{s:aa}
The Anderson acceleration, or Anderson mixing, was initially developed in 1965 by Donald Anderson \cite{Ande65} as an iterative procedure for solving some  nonlinear integral equations arising in physics. It turns out that the Anderson acceleration is very efficient for solving other types of nonlinear equations as well, see \cite{FaSa08}, \cite{WaNi10}, and the literature cited therein. In \cite{WaNi10} it was shown that on fixed point linear problems the Anderson acceleration, with all mixing parameters equal to 1, and \\GMRES are ``essentially equivalent". In the present paper we extend the results of \cite{WaNi10} for general linear problems and general nonzero mixing parameters. By introducing the notion of index of the Anderson acceleration, $\kappa_A$, we manage to give a complete characterization of the behavior of the Anderson acceleration with infinite history on linear problems. We show that the index of the Anderson acceleration is the same for any choice of nonzero mixing parameters, and  that it can be defined in terms of the stagnation index of the GMRES method. The main result of the paper shows that if the index of the Anderson acceleration coincides with the grade of the residual at the starting point with respect to the matrix defining the linear problem, $\nu(A,r_0)$ \cite[pp 37]{Wilk65}, then the Anderson acceleration converges to the exact solution of the linear problem in either  $\nu(A,r_0)$ or  $\nu(A,r_0)+1$ steps. If  $\kappa_A<\nu(A,r_0)$, then the Anderson acceleration converges to the wrong solution. We also investigate the optimal Anderson acceleration, where at each step the mixing parameter is chosen so that it minimizes the residual of the current iterate. We show that the performance of the optimal Anderson acceleration is not essentially  better than the performance of the  Anderson acceleration with arbitrary nonzero mixing parameters.
\section{{Simple Mixing}}
Consider the linear equation
\begin{equation} \label{A}
Ax + b = 0
\end{equation}
where $A$ is a nonsingular $N \times N$ matrix and $b$ is a given $N$ vector. We wish to solve (\ref{A}) with various iterative methods that produce sequences $x_n^{[M]}$ where the superscript $[M]$ indicates the method that is being used. Let $x^* = -A^{-1}b$ be the exact solution of this problem. Since the exact solution is usually not known,the errors  $x_n^{[M]} - x^*$  are difficult to estimate, so that the performance of the method is assessed by analyzing the residuals $Ax_n^{[M]} + b$. All methods use the same starting point $x_0^{[M]}=x_0$ at which the residual is $r_0=Ax_0 + b$.

\medskip
Consider now the simple fixed point iteration
\begin{equation}\label{fp}
x_{n+1}= x_n + Ax_n + b = Mx_n + b
\end{equation}
where $M = I+A$. Of course this scheme need not converge, and an improvement consists in iterating
\begin{equation} \label{mixx}
x_{n+1}^S = (1 - \beta) x_n^S + \beta M \left(x_n^S + b \right) = x_n^S + \beta \left(A x_n^S + b\right)
\end{equation}
where $\beta$ is a suitably chosen parameter. The method averages or ``mixes'' the previous iterate and the new fixed point iterate.  We shall call this iteration \emph{simple mixing} and indicate it with the superscript $S$.

\medskip

\section{{GMRES and Anderson mixing (Anderson acceleration)}}
The GMRES method for the equation $Ax+b = 0$  determines $x_n^G$ as
\begin{equation}\label{gmres}
x_n^G=x_0+z_n,\quad\mbox{\rm where,   }\quad z_n=\argmin_{z\in\cK_n}\{\|A (x_0+z)+b\|\; :\; z\in \mathcal{K}_n\}.
\end{equation}
Here $\mathcal{K}_n$ is the Krylov space
\begin{equation} \label{krylov}
\mathcal{K}_n=\cK_n(A,r_0) = Span \, \{ r_0, Ar_0, \dots, A^{n-1}r_0 \} \, .
\end{equation}
Note that since $r_0=Ax_0+b$,  $x_1^G$ is given by
$x_1^G = x_0 + \beta^*(Ax_0+b)$, where
\begin{equation}\label{beta*}
\beta^* = \argmin_\beta \|r_0 + \beta Ar_0\| = -\frac{r_0^TAr_0}{\|Ar_0\|^2}  \, ,
\end{equation}
that is, the first step of GMRES is a simple mixing step with a mixing parameter $\beta$ that minimizes the residual of the result.

Let $K_n$ denote the projection onto the subspace $A\mathcal{K}_n$ of $\R^N$. From (\ref{gmres}) it follows that
\begin{equation}\label{GP0}
A(x_n^G-x_0)=\, Az_n=- K_n(Ax_0+b)=-K_nr_0\, .
\end{equation}
Therefore we can write
\begin{equation}\label{GP}
x_n^G-x^*=x_0-x^*-A^{-1}K_nr_0=A^{-1}(I-K_n)r_0 \, ,
\end{equation}
which is equivalent to
\begin{equation}\label{GP1}
Ax_n^G+b=(I-K_n)r_0 \, .
\end{equation}
We also note that
\begin{eqnarray}
&&\mathcal{K}_n=\mathcal{K}_{n+1}\Leftrightarrow
(I-K_n)A^{n+1}r_0=0\, ,\label{kn1}\\
&&
x^G_n=x^G_{n+1}\Leftrightarrow K_nr_0=K_{n+1}r_0\Leftrightarrow
r_0^T(I-K_n)A^{n+1}r_0=0\, .\label{xn1}
\end{eqnarray}
Following Wilkinson \cite[pp 37]{Wilk65}, the \emph{grade} of $r_0 \ne 0$ with respect to $A$ is defined as
\begin{equation}\label{grade}
\nu(A,r_0) = 1+ \max \, \{ n \; :\; dim \, \cK_n = n \}
\end{equation}
Thus $\nu(A,r_0)$ is the smallest integer $n$ for which there is a non-zero polynomial $p(z)$ of degree $n$ such that $p(A)r_0 = 0$, i.e.,
\begin{equation}\label{grade1}
\nu(A,r_0) = \min\{n\in\N\; : \; r_0, Ar_0,\ldots, A^{n}r_0 \;\mbox{are linearly dependent}\,\} .
\end{equation}
Then clearly $\nu(A,r_0) \le N$, and $p(z)$ divides the minimal polynomial of $A$. Also,
$
\cK_n = \cK_\nu(A,r_0),\;\forall n \ge \nu = \nu(A,r_0)
$.

\begin{proposition}\label{p:kG}
The GMRES method (\ref{gmres}) converges in exactly $\nu(A,r_0)$ steps, i.e.,
\[
x_n^G\neq x^*,\; for\;  n<\nu(A,r_0), \; and \; x_n^G=x^*,  for\;  n\ge \nu(A,r_0)\, .
\]
\end{proposition}
\begin{proof}
If $n = \nu(A,r_0)$, then from (\ref{grade1} it follows that there are numbers $\xi_0,\ldots,\xi_n$ such that
$\sum_{i=0}^{n}\xi_iA^ir_0=0$. It is easily seen that we must have $\xi_0\neq 0$ and $\xi_{n}\neq 0$, because otherwise the minimality of $n$ is contradicted. Since $\xi_0\neq 0$, we deduce that $r_0 \in A\mathcal{K}_{n}$. According to (\ref{GP}) we have therefore $x^G_{n}=x^*$. If $n<\nu(A,r_0)$, then the vectors $ r_0,Ar_0,\ldots,A^{n-1}r_0$ are linearly independent, so that $r_0 \notin A\mathcal{K}_{n}$, which means that $x_n^G\neq x^*$.
\end{proof}

We next summarize Anderson mixing, for the nonlinear equation $f(x) = 0$.

\bigskip

\noindent{\it Anderson mixing.} Given a nonlinear operator $f$ on $\R^N$, an initial point $x_0 \in \R^N$, a sequence $\beta_0, \, \beta_1, \dots $ in  $\R \setminus \{0\}$,  and an integer $m$:

\smallskip
0. Set
\begin{eqnarray*}
f_0=f(x_0),\quad x_1=x_0+\beta_0 f_0\, .
\label{sta}
\end{eqnarray*}

\smallskip
1. For $k=1,2,\ldots$, set
\begin{eqnarray*}
&&m_k=\min\{m,k\},\quad r_k = k-m_k,  \quad  f_k=f(x_k),\\
&&(\alpha_{0,k},\ldots,\alpha_{m_k,k})=
\argmin_{(\alpha_0,\ldots,\alpha_{m_k})} \{\|\sum_{i=0}^{m_k}\alpha_i f_{r_k+i}\|^2
\; : \; \sum_{i=0}^{m_k}\alpha_i =1\},\\
&&x_{k+1}=\sum_{i=0}^{m_k}\alpha_{i,k} x_{r_k+i} + \beta_k \sum_{i=0}^{m_k}\alpha_{i,k}f_{r_k+i} =\sum_{i=0}^{m_k}\alpha_{i,k}( x_{r_k+i} +\beta_k f_{r_k+i})\, .
\end{eqnarray*}

\bigskip

For $\beta_i\equiv 1$ and $f(x)=g(x)-x$, this algorithm reduces to Algorithm AA from \cite{WaNi10}. The version given here was proposed in  \cite{Ande65}. There is also an equivalent (in exact arithmetic) version of this algorithm in terms of difference vectors $x_{k} - x_{k-1}$ and $f(x_k) - f(x_{k-1})$. It is presented in \cite{FaSa08} in order to reveal its connection to multisecant and Broyden type methods for solving nonlinear operator equations, but it obscures the ``mixing'' idea. We therefore do not give that version here.

\bigskip

If we solve the constrained optimization problem from this algorithm  with the substitution method, e.g. by setting $\alpha_0=1-\sum_{i=1}^{m_k}\alpha_i$ and solving the corresponding unconstrained optimization problem for $\alpha_1,\ldots,\alpha_{m_k}$,
\begin{equation}\label{unc}
\min_{\alpha_1,\ldots,\alpha_{m_k}}\|f_{r_k}+\alpha_1(f_{r_k+1}-f_{r_k})+\alpha_2(f_{r_k+2}-f_{r_k})
+\ldots+\alpha_{m_k}(f_{k}-f_{r_k})\|\, ,
\end{equation}
then the differencing may lead to problems with loss of significance.
However, the authors of  \cite{WaNi10} claim that implementing the Anderson acceleration with a substitution method ``offers several advantages and, in our
experience, no evident disadvantages''.
Other methods for solving the constrained optimization problem  are discussed in \cite{Ni09}.

\section{Convergence analysis in the linear case}

In what follows we will investigate the relationship between the Anderson mixing and GMRES, for linear systems of the form
\ref{A}. We use the notation  $x_n^G$ for the sequence generated by the GMRES method as described in the previous section, and $x_n^A$ for the sequence generated by Anderson mixing  with $m=\infty$  and $f(x)=Ax+b$, i.e,
\begin{eqnarray}\label{a}
&& \quad x_{n+1}^A=\sum_{i=0}^{n}\alpha_{i,n}x_i^A +\beta_n\sum_{i=0}^{n}\alpha_{i,n}(A x_i^A +b)= \bar{x}_{n+1}^A + \beta_n (A \bar{x}_{n+1}^A +b) ,\\
&& \quad(\alpha_{0,n},\ldots,\alpha_{n,n})=
\argmin_{(\alpha_0,\ldots,\alpha_{n})} \{\|b+A\sum_{i=0}^{n}\alpha_i  x_i^A\|\; : \; \sum_{i=0}^{n}\alpha_i=1\}.\nonumber
\end{eqnarray}

\noindent Here
\begin{equation}\bar{x}_{n+1}^A =\sum_{i=0}^{n}\alpha_{i,n}x_i^A \label{barx}
\end{equation}
may be viewed as a prediction of the next iterate, a linear combination of all previous iterates. Using $\sum_{i=0}^n\alpha_i=1$, we deduce that
\begin{equation}\label{barx1}
\bar{x}_{n+1}^A =(1-\sum_{i=1}^{k}\alpha_{i,n})x_0+\sum_{i=1}^{n}\alpha_{i,n} x_i^A=
x_0+\sum_{i=1}^{n}\alpha_{i,n} (x_i^A-x^0).
\end{equation}
Therefore (\ref{a}) can be written as
\begin{eqnarray}
&&\quad x_{n+1}^A=x_0+\sum_{i=1}^{n}\alpha_{i,n}(x_i^A-x_0)+\beta_n(b+ A(x_0+\sum_{i=1}^{n}\alpha_{i,n}(x_i^A-x_0)),
\label{aa}\\
&&\quad(\alpha_{1,n},\ldots,\alpha_{n,n})=
\argmin_{(\alpha_1,\ldots,\alpha_{n})} \|b+A(x_0+\sum_{i=1}^{n}\alpha_i  (x_i^A-x^0))\|\, .\label{aa1}
\end{eqnarray}
Let us consider the linear subspace
\begin{equation}\label{lsp}
\mathcal{L}_n=Span\{x_1^A- x_0,\ldots, x_n^A-x_0\}\, ,
\end{equation}
and denote by $L_n$ the projection onto $A\mathcal{L}_n$. With this notation we have
\begin{equation}\label{axb}
A(\bar{x}_{n+1}^A-x_0)=-L_n(Ax_0+b)=-L_nr_0,
\end{equation}
which is equivalent to
\begin{equation}\label{xba}
\bar{x}_{n+1}^A=x_0-A^{-1}L_nr_0=x^*+A^{-1}r_0-A^{-1}L_nr_0=x^*+A^{-1}(I-L_n)r_0\, .
\end{equation}
Therefore,
\begin{eqnarray}\label{aab}
&&\quad A \bar{x}_{n+1}^A +b=(I-L_n)r_0\, ,
\\&&\quad
\label{aal}
x_{n+1}^A-x^*=(I+\beta_n A)A^{-1}(I-L_n)r_0=(I+\beta_n A)(\bar{x}_{n+1}^A-x^*)\, ,
\\&&\quad
\label{aal2}
Ax_{n+1}^A+b=(I+\beta_n A)(I-L_n)r_0=(I+\beta_n A)(A\bar{x}_{n+1}^A+b)\, ,
\\&&\quad
\label{aal3}
A(x_{n+1}^A-x_0)=-L_nr_0+\beta_n A(I-L_n)r_0=\beta_nAr_0-L_nr_0-\beta_n AL_nr_0
\, .
\end{eqnarray}

\begin{definition}\label{d:kA} The index of the Anderson acceleration  (\ref{aa}) is defined as
\begin{equation}\label{dka}
\qquad \kappa_A=\min\{n\in\N\; : \; x^A_1-x_0,x^A_2-x_0,\ldots,x^A_{n+1}-x_0 \;\mbox{are linearly dependent}\,\} .
\end{equation}
The stagnation index of the GMRES method (\ref{gmres}) is defined as
\begin{equation}\label{stag}
\qquad \eta^G=\min\{n\in\{0,1,\ldots\}\; : \; x^G_n=x^G_{n+1}\,\} .
\end{equation}
\end{definition}
The above notion allows for a complete description of the convergence properties of the Anderson acceleration for linear problems, for arbitrary sequences of relaxation parameters $\beta_n$.
\medskip

\begin{proposition}\label{p:a} The index $\kappa_A$ of the Anderson acceleration is always less than or equal than the grade $\nu(A,r_0)$, and the sequences generated by the two methods satisfy the following properties:\\
\begin{enumerate}
\item[a)]
$\;x^A_{n+1}=x^G_{n}+\beta_n(Ax^G_n +b),\;n=0,1,\ldots,\kappa_A$ ;\\
\item[b)] $\bar{x}_{n+1}^A=\;x^G_n,\;n=0,1,\ldots,\kappa_A$.\\
\end{enumerate}
\end{proposition}
\begin{proof}
Let us first prove by induction that $\mathcal{L}_n\subset\mathcal{K}_n$. Since $x_1^A- x_0=\beta_0 r_0$ this is readily verified for $n=1$. Assume that $\mathcal{L}_n\subset\mathcal{K}_n$. Then from (\ref{aa}) it follows that
\begin{equation}\label{xa}
x^A_{n+1}-x_0=\beta_n r_0+(I+\beta_n A)\sum_{i=1}^{n}\alpha_{i,n}(x_i^A-x_0))\in \mathcal{K}_n+ A \mathcal{K}_n=\mathcal{K}_{n+1},
\end{equation}
which completes the induction step. Since $\mathcal{L}_n\subset\mathcal{K}_n$, the linear independence of $x_1^A- x_0,\ldots, x_n^A-x_0$ always implies the linear independence of $r_0,\ldots,A^{n-1} r_0$. This implies that $\kappa_A\le \nu(A,r_0)$. It also implies that
\begin{equation}\label{dim}
\cL_n = \cK_n\, ,\;\mbox{\rm and } \;dim(\mathcal{K}_n)=n,\;\mbox{\rm for  }\;n=1,\ldots,\kappa_A\, .
\end{equation}

Point a) of our proposition follows from (\ref{GP}) and
(\ref{aal}), while point b) follows from (\ref{GP}) and (\ref{xba}).
\end{proof}

\smallskip
The result shows that varying the sequence of relaxation parameters $\beta_k$ does not change the behavior of the sequence $x_k^A$ in a significant way for linear problems, if exact arithmetic is used and as long as $k \le m$. However, for nonlinear problems, the convergence behavior can be quite sensitive to the ``right'' choice of the $\beta_k$; see e.g. \cite{MaLu08}. Also, the choice of the $\beta_n$ does matter for large $n$ if $m$ is finite, as is shown by numerical experience.

\smallskip
It is instructive to observe what exactly happens at the step $\kappa_A$. Of course, if $\kappa_A = \nu(A,r_0)$, the Anderson mixing has converged to the correct solution. In the case where $\kappa_A < \nu(A,r_0)$, we have the following characterization.
\smallskip
\begin{proposition}
\label{p:aa}
If $n < \kappa_A < \nu(A,r_0)$, then\\
\begin{enumerate}
\item[a)]
$\;\alpha_{n,n} \ne 0\;$ in (\ref{barx1}), \\
\item[b)]
$\;\|Ax_{n-1}^G+b\| > \|Ax_n^G+b\|\;$ .\\
\end{enumerate}
If $n = \kappa_A< \nu(A,r_0)$, then\\
\begin{enumerate}
\item[c)]
$\;\alpha_{n,n} = 0\;$ in (\ref{barx1}), \\
\item[d)]
$\; (A\bar{x}_{n}^A +b)^TA(x_{n+1}^A - x_0) = 0 \; $, \\
\item[e)]
$\;x^G_{n-1}=\bar{x}_n^A = \bar{x}_{n+1}^A = x^G_n= \bar{x}_{n+2}^A \;$ ,\\
\item[f)]
$\;\|Ax^G_{n-1}+b\|=\|Ax^G_n+b\|\;$ .\\
\end{enumerate}
\end{proposition}

\smallskip
\begin{proof}
Using (\ref{dim}) we can write
\begin{equation}\label{xi}
x^A_n-x_0=\sum_{i=1}^n\xi_{i,n}A^{i-1}r_0,\quad \mbox{\rm with } \xi_{n,n}\neq 0\, ,\quad n=1,\ldots,\kappa_A
\end{equation}
for some  uniquely determined scalars $\xi_{1,n},\ldots,\xi_{n,n}$.
Using (\ref{aa}) we obtain
\begin{eqnarray*}
x_{n+1}^A - x_0
&=&\beta_nr_0+\sum_{i=1}^{n}\alpha_{i,n}(x_i^A-x_0)+\beta_nA\sum_{i=1}^{n-1}\alpha_{i,n}(x_i^A-x_0)
+\beta_n\alpha_{n,n}A(x_n^A-x_0)
\\
&=&y_n+\beta_n\alpha_{n,n}\xi_{n,n}A^{n}r_0,\quad \mbox{\rm with } y_n\in\cL_n=\cK_n\, ,\quad n=1,\ldots,\kappa_A\,.
\end{eqnarray*}
If $n<\kappa_A$, then $\cL_{n+1}\neq\cL_{n}$, so that we must have $\alpha_{n,n} \ne 0$. This proves part a).
On the other hand, if $n=\kappa_A<\nu(A,r_0)$, then $\cL_{n+1}=\cL_{n}=\cK_{n}\neq\cK_{n+1}$. This implies $\alpha_{n,n} = 0$, which shows the validity of part c).
If $\alpha_{n,n} = 0$, then from (\ref{aa}) and (\ref{barx1}) we have
$
\bar{x}_n^A = \bar{x}_{n+1}^A \, ,
$ which, in view of Proposition~\ref{p:a}, implies $\;x^G_{n-1}=\bar{x}_n^A = \bar{x}_{n+1}^A = x^G_n$. Since $\cL_{n+1}=\cL_{n}$ we also have $ \bar{x}_{n+1}^A =  \bar{x}_{n+2}^A$, which completes the proof of part e). Part f) follows trivially.

\smallskip

To prove d), observe that the normal equations from (\ref{aa}) can be written as
\[
\left( A \bar{x}_{n+1}+b\right)^TA(x_i^A-x_0) = 0\, , \quad i = 1, \dots, n \, .
\]
Geometrically, the above relation follows immediately, since according to (\ref{aab}),\\ $A \bar{x}_{n+1}+b\in\left( A\mathcal{L}_n\right)^\perp$.
In case $n=\kappa_A$, according to e), we have $\bar{x}_{n+1}=\bar{x}_{n}$. For $i=n$ we get part d).

\smallskip

To prove part b), note that by (\ref{gmres}), $\|Ax_n^G+b\| \le \|A x_{n-1}^G+b\|$. Suppose there is equality, then from (\ref{GP1}) $\|(I - K_n) r_0\| = \|(I - K_{n-1})r_0 \|$. Since these are orthogonal projections, this implies in turn $K_nr_0 = K_{n-1}r_0$, and therefore by Proposition~\ref{p:a}, (\ref{GP}), and (\ref{xba}), we should have
$
\bar{x}_{n+1}^A = x_n^G  = x_{n-1}^G = \bar{x}_n^A \,
$.
But this is impossible, since $\bar{x}_n^A - x_0 \in \cL_{n-1}$, but $\bar{x}^A_{n+1} - x_0 \in \cL_n \setminus \cL_{n-1}$, due to $\alpha_{n,n} \ne 0$. Therefore, $\|Ax^G_n+b\| < \|Ax^G_{n-1}+b\|$.
\end{proof}

\bigskip
Using propositions \ref{p:kG}, \ref{p:a}, and \ref{p:aa}, we obtain the following theorem, which represents the main result of our paper.

\smallskip
\begin{theorem}\label{t:aa}
Assume that $A$ is an invertible $N\times N$-matrix. Consider the GMRES method (\ref{gmres}), and the Anderson acceleration method (\ref{aa})  for finding the solution $x^*$ of the linear equation (\ref{A}), with arbitrary nonzero mixing parameters $\beta_0,\beta_1,\ldots$. For any starting point $x_0$,  consider also the grade $\nu(A,r_0)$,  defined in (\ref{grade1}), and the index of the Anderson acceleration $\kappa_A$, defined in  (\ref{dka}).

\medskip

(i)  $\kappa_A=\nu(A,r_0)=\nu$ if and only if
\begin{equation}\label{i1}
\|Ax_0+b\|>\|Ax_1^G+b\|>\cdots>\|Ax_{\nu-1}^G+b\|>\|Ax_{\nu}^G+b\|=0\, .
\end{equation}
Moreover, in this case the sequence produced by the Anderson acceleration satisfies
\begin{eqnarray}
x^A_{n+1}&=&\left\{\begin{array}{cl}x^G_n+\beta_n(Ax^G_n+b),& \mbox{\rm if   }\; n<\nu\, \\
x^*,&\mbox{\rm if   }\; n\ge\nu\end{array}\right . ,\label{i2}\\
\bar{x}^A_{n+1}&=&x^G_n,\quad n=0,1,\ldots\; \, .\label{i3}
\end{eqnarray}

\medskip

(ii) $\kappa=\kappa_A<\nu(A,r_0)$ if and only if
\begin{equation}\label{ii1}
\|Ax_0+b\|>\|Ax_1^G+b\|>\cdots>\|Ax_{\kappa-1}^G+b\|=\|Ax_{\kappa}^G+b\|>0\, .
\end{equation}
Moreover, in this case $\kappa_A = \eta^G + 1$, and the sequence produced by the Anderson acceleration satisfies
\begin{eqnarray}
x^A_{n+1}&=&\left\{\begin{array}{cl}x^G_n+\beta_n(Ax^G_n+b),& \mbox{\rm if   }\; n\le\kappa\, \\
x^G_{\kappa-1}+\beta_n ( Ax^G_{\kappa-1}+b),&\mbox{\rm if   }\; n\ge\kappa\end{array}\right .\, ,\label{ii2}\\
\bar{x}^A_{n+1}&=&\left\{\begin{array}{cl}x^G_n,& \mbox{\rm if   }\; n\le\kappa\, \\
x^G_{\kappa-1},&\mbox{\rm if   }\; n\ge\kappa\end{array}\right .\, .\label{ii3}
\end{eqnarray}

\end{theorem}
\begin{proof}

\noindent The fact that $\kappa_A=\nu(A,r_0)=\nu$ implies (\ref{i1}), (\ref{i2}), and  (\ref{i3}) follows from propositions \ref{p:kG}, \ref{p:a} and \ref{p:aa}.

\medskip

\noindent In order to prove that  $\kappa=\kappa_A<\nu(A,r_0)$ implies (\ref{ii1}), we first  prove by induction that $\mathcal{L}_n=\mathcal{L}_{\kappa}=\mathcal{K}_{\kappa}$ for all $n>\kappa$. This is certainly true for $n=\kappa+1$ from Definition~\ref{d:kA}. Suppose that our statement is true for an $n>\kappa$. Then, according to (\ref{GP}), (\ref{xba}), and point e) of Proposition~\ref{p:aa},  we have
$
\bar{x}^A_{n+1}= x_{\kappa}^G=x_{\kappa-1}^G\,
$.
This proves $x^A_{n+1}=x^G_{\kappa-1}+\beta_n ( Ax^G_{\kappa-1}+b)$. To complete the induction step, we note that from (\ref{GP0}) and(\ref{GP1}) it follows that
\[
A(x^A_{n+1}-x_0)=A(x^G_{\kappa-1}-x^0)+\beta_n A( Ax^G_{\kappa-1}+b)=\beta_nAr_0-(I+\beta_nA)K_{\kappa-1}r_0\, .
\]
Since $A K_{\kappa-1}r_0\in A^2\cK_{\kappa-1}\subset\cK_{\kappa+1}$, we deduce that $x^A_{n+1}-x_0\in\cK_{\kappa}$, which shows that $\mathcal{L}_{n+1}=\mathcal{K}_{\kappa}$.
The fact $\kappa=\kappa_A<\nu(A,r_0)$ implies (\ref{ii2})  and (\ref{ii3}) and that $\kappa_A = \eta^G$ is this case  follows from propositions \ref{p:a} and \ref{p:aa}.

Assume now that (\ref{i1}) holds, but   $\kappa_A\neq\nu(A,r_0)=\nu$. Since $\kappa_A\le\nu(A,r_0)$, this implies $\kappa=\kappa_A<\nu(A,r_0)$, which in turn, as seen above implies (\ref{ii1}), so that (\ref{i1}) cannot be true. Hence,  (\ref{i1}) is equivalent to $\kappa_A=\nu(A,r_0)=\nu$.

Similarly if $\kappa=\kappa_A<\nu(A,r_0)$ is not true, we must have $\kappa=\kappa_A=\nu(A,r_0)$, which, as seen above, implies (\ref{i1}) so that (\ref{ii1}) is not true. Therefore (\ref{ii1}) is equivalent to $\kappa=\kappa_A<\nu(A,r_0)$.
\end{proof}

\smallskip
Theorem \ref{t:aa} gives a complete characterization of the behavior of the Anderson acceleration on linear problems.
If $\kappa_A=\nu(A,r_0)$, then the Anderson acceleration converges to $x^*$. If $\kappa_A<\nu(A,r_0)$, then
$\kappa_A$ is precisely the first index for which \\ GMRES stagnates (i.e. produces two identical successive iterates). If this ever happens, GMRES continues to generate larger Krylov spaces, and it will eventually converge to $x^*$, while Anderson mixing will then stagnate forever.  An extreme example is given by $A = P_N$, where $P_N$ is the permutation matrix for the cycle $(123\dots N)$, with minimal polynomial $q(z) = 1 - z^N$. Then if $b$ is any standard basis vector $e_k$, Anderson mixing will immediately stagnate (i.e. $\kappa_A = 1$), while GMRES will converge in $\nu(A,b) = N$ steps, stagnating at the initial value until the very last step.

\smallskip
\begin{corollary}
The Anderson acceleration for linear problems converges in at most $\kappa_A+1$ steps, but not necessarily to the solution of the linear problem. If $\kappa_A = \nu(A,r_0)$, the Anderson acceleration  converges to the exact solution of the linear problem in either $\nu(A,r_0)$  or  $\nu(A,r_0)+1$ steps.
\end{corollary}
\begin{proof}
The first part of the Corollary follows directly from Theorem~\ref{t:aa}. We note that if $n< \kappa_A$, then
$
x^A_{n+1}-x^*=x^G_n-x^*+\beta_n(Ax^G_n+b)=x^G_n-x^*+\beta_nA(x^G_n-x^*)\, .
$
Therefore $x^A_{n+1}-x^*$ if and only if $x^G_n-x^*$ is an eigenvector of $A$ with eigenvalue $-1/\beta_n$,  but in this case we have $\|Ax^G_{n+1}+b\|\le\|Ax^A_{n+1}+b\|=0$, so that $x^G_{n+1}=x^*$, and therefore $n+1=\nu(A,r_0)$.
\end{proof}

\smallskip
\begin{corollary}
We have $\eta^G=0, \kappa_A = 1$ if and only if the quantity $\beta^*$ defined in (\ref{beta*}) vanishes.
\end{corollary}
\begin{proof}
If $\beta^* = 0$, then (as noted in the remarks following (\ref{beta*})), $x_1^G = x_0 + \beta^*(Ax_0+b) = x_0$, and hence $\eta^G = \kappa_A - 1 = 0$.
\end{proof}

\section{Anderson acceleration with optimized mixing parameters}
The Anderson acceleration for the linear problem (\ref{A}) defined in (\ref{aa}) depends on a sequence of mixing parameters $\beta_0,\beta_1,\ldots$. In this section we consider a variant of the Anderson acceleration, where at each step $\beta_n$ is chosen so that the residual at $x^A_{n+1}$ is minimal. More precisely we consider the following algorithm:

\medskip

\noindent{\it Optimized Anderson acceleration for linear problems.}\\
\noindent Set $\bar{x}^{A*}_1=x_0$ and ${x}^{A*}_1=x_0+\beta_0^*(Ax_0+b)$, with $\beta^*$ defined in (\ref{beta*});\\
\noindent For $n=1,2,\ldots$

Compute
$$
(\alpha_{1,n}^*,\ldots,\alpha_{n,n}^*)=
\argmin_{(\alpha_1,\ldots,\alpha_{n})} \|b+A(x_0+\sum_{i=1}^{n}\alpha_i  (x_i^{A*}-x^0))\|\, ;
$$

Compute
$$
\bar{x}_{n+1}^{A*} = x_0+\sum_{i=1}^{n}\alpha_{i,n}^* (x_i^{A*}-x^0)
$$

If
$A\bar{x}^{A*}_{n+1}+b=0$, set $\beta_n^*=0$. Otherwise set
$$
\beta_n^*=-\,
\frac{(A\bar{x}^{A*}_{n+1}+b)^TA(A\bar{x}^{A*}_{n+1}+b)}{\|A(A\bar{x}^{A*}_{n+1}+b)\|^2};
$$

Set $x^{A*}_{n+1}=\bar{x}^{A*}_{n+1}+\beta_n^*(A \bar{x}^{A*}_{n+1}+b)$.

\bigskip

We note that Theorem~\ref{t:aa} implies that $\kappa_A$ is the same for any choice of nonzero mixing parameters $\beta_0,\beta_1,\ldots$ and that the sequence $\bar{x}_n^A$ is independent of this choice. Also, once $\beta_n^* = 0$, then clearly $\beta_r^* = 0$ for all $r > n$. Therefore, if $\beta^*_n \ne 0$ for some $n \ge 0$, then
\[\beta^*_r \ne 0, \quad \bar{x}_{r+1}^A = \bar{x}_{r+1}^{A*} \label{barx*}
\]
for all $r \le n$. However, it seems possible that the $\beta_n^*$ become zero for some $n < \kappa_A$ and that optimized Anderson acceleration stagnates before a general Anderson acceleration scheme stagnates (in which $\beta_n \ne 0$ for all $n$ is enforced).  We now show that this can never happen and that $\beta^*_n$ becomes zero precise for $n > \kappa_A$.

\begin{theorem}\label{t:ao}
Assume that $A$ is an invertible $N\times N$-matrix. Consider the \\
GMRES method (\ref{gmres}), the Anderson acceleration method (\ref{aa}), with arbitrary nonzero mixing parameters $\beta_0,\beta_1,\ldots$, and the  optimized Anderson acceleration for linear problems described above. Then for all $n$, $\beta_n^* \ne 0$ if and only if $n \le  \kappa_A = \eta^G +1$. Also for all $n$,
\begin{eqnarray}
&&x^{A*}_{n+1}=\left\{\begin{array}{cl}x^G_n+\beta_n^*(Ax^G_n+b),\;
& \mbox{\rm if   }\; n<\eta^G\, \\
x^G_{\eta^G},&\mbox{\rm if   }\; n\ge\eta^G\end{array}\right . ,\label{oa1}\\
&&\bar{x}^{A*}_{n+1}= \bar{x}^A_{n+1} = \left\{\begin{array}{cl}x^G_n,& \mbox{\rm if   }\; n<\eta^G\, \\
x^G_{\eta^G},&\mbox{\rm if   }\; n\ge\eta^G\end{array}\right . ,\label{oa1a}\\
&&\|Ax^{G}_{n+1}+b\|\le\|Ax^{A*}_{n+1}+b\|<\|Ax^{G}_{n}+b\|,\quad \mbox{\rm if   }\; n<\eta^G\,\label{oa2}\\
&&\|A x_{n+1}^{A*}+ b\|^2 = \|A\bar x^{A*}_{n+1} + b\|^2 - \beta_n^{*2} \|A(A \bar x_{n+1}^{A*} + b)\|^2
\label{oa3a}\\
&&\hspace{.88in} \le \|A x^{A*}_n + b\|^2 - \beta_n^{*2} \|A(A \bar x_{n+1}^{A*} + b)\|^2
\label{oa3b}
\end{eqnarray}
\end{theorem}

\medskip
\begin{proof}
We begin by observing that (\ref{oa3a}) and (\ref{oa3b}) follow from the construction of the $x_n^{A*}$ for all $n$. Let
\begin{equation}
R = \max \{ n \, | \, \beta_n^* \ne 0 \} \label{N} \, .
\end{equation}

Then $\beta_n^* \ne 0$ for all $n < R$,  since otherwise $x_n^{A*} = x_{n+1}^{A*}$ for some $n < R$ and hence $\beta_k^* = 0$ for all $k \ge n$.  Then an induction argument and the results in Theorem~\ref{t:aa}
show that (\ref{oa1}) -- (\ref{oa2}) hold for $n \le R$. Clearly, $R \le \kappa_A$. It remains to show that
$R \ge \kappa_A$.

\smallskip
Let therefore $n = R+1$, and assume that $n \le \kappa_A$. We may also assume that $n \le \nu(A,r_0)$, since otherwise there is nothing to prove. We want to show that
\begin{equation}
\bar x_{n+1}^A = \bar x_{n+2}^A \label{n+2}
\end{equation}
which will imply the direct contradiction $n = R+1 > \kappa_A$. Now the equation $A\bar x^A_{n+1} - A x_0 = -L_n r_0$ (see (\ref{axb})), where $L_n$ is the orthogonal projection on $A\cL_n = A \cL_n$, is equivalent to
\begin{equation} \xi^T (A\bar x_{n+1}^A + b) = 0  \label{kn}
\end{equation}
for all $\xi \in A \cL_n$. Showing (\ref{n+2}) is now equivalent to proving (\ref{kn}) for all $\xi \in A \cL_{n+1}$, since then $ A\bar x^A_{n+1} - A x_0 = -L_{n+1} r_0 = A\bar x^A_{n+2} - A x_0 $ and hence (\ref{n+2}) holds. To show (\ref{kn}) for all $\xi \in A \cL_{n+1}$, it suffices to show (\ref{kn}) for a single $\xi \in A \cL_{n+1} \setminus A \cL_n$. Now use the assumption $\beta_n^* = 0$, which is equivalent to
\[ \left(A(A\bar x_{n+1}^A + b)\right)^T (A\bar x_{n+1}^A + b) = 0 \, .
\]
We wish to show that
\begin{equation}
\xi = \left(A(A\bar x_{n+1}^A + b)\right) \in A\cL_{n+1} -  A \cL_n \label{kn+1}
\end{equation}
which will prove (\ref{kn}) for all $\xi \in A\cL_{n+1}$ and complete the proof of the theorem.

\smallskip
\noindent
We know that $\bar x^A_{n+1} - x_0 \in \cL_n$. But $\bar x^A_{n+1} - x_0 \notin \cL_{n-1} = \cL_R$, since otherwise
\[A\bar x^A_{n+1} - A x_0 = L_{n-1} \left(A\bar x^A_{n+1} - A x_0\right)  = -L_{n-1} L_n r_0 = - L_{n-1}r_0
\]
and therefore already $\bar x^A_R = \bar x^A_{R+1}$, contradicting $R < \kappa_A$. Therefore we can write
\[\bar x^A_{n+1} - x_0 = \sum_{j=0}^{n-1} \lambda_j A^j r_0 \in \cL_n
\]
and $\lambda_{n-1} \ne 0$. Consequently
\[A\bar x^A_{n+1} + b = r_0 + \sum_{j=0}^{n-1} \lambda_j A^{j+1} r_0 \in \cL_{n+1}
\]
But if $A\bar x^A_{n+1} + b \in \cL_n$, then a calculation shows that $A\bar x^A_{n+1} - Ax_0  \in A\cL_{n-1}$ and hence $\bar x^A_{n+1} - x_0 \in \cL_{n-1}$, which was ruled out above. Consequently
\[A\bar x^A_{n+1} + b \in \cL_{n+1}\setminus \cL_n \, .
\]
This implies (\ref{kn+1}) and therefore (\ref{kn}) for all $\xi \in A \cL_{n+1}$. Then $\bar{x}^A_{n+2} = \bar{x}^A_{n+1}$,  and all conclusions of the theorem follow.
\end{proof}

\medskip
The theorem shows that optimized Anderson acceleration enjoys the descent property
\[\|A x^{A*}_{n+1} + b\| < \|A \bar x^{A*}_{n+1} + b\| \le \|A x^{A*}_n + b\|
\]
for all $n < \kappa_A$, but that it does not accelerate convergence otherwise when compared to Anderson acceleration with arbitrary $\beta_n \ne 0$.

\bibliographystyle{plain}
\bibliography{mi}

\end{document}